\documentclass[leqno,12pt]{amsart} 
\usepackage{changepage}
\usepackage{adjustbox}

\usepackage{amsfonts,amsmath,amssymb,amsthm}

\usepackage[all]{xy}
\xyoption{matrix}
\usepackage[svgnames]{xcolor}
\usepackage[normalem]{ulem}
\usepackage{mathabx} 
\newcommand{\nicecolor}{Navy}  
\usepackage[shortlabels]{enumitem}
\setlist[1]{wide}
\setlist[2]{leftmargin=15mm} 
\setlist[enumerate]{label=\rm{(\arabic*)}}
\setlist[enumerate,2]{label=\rm({\it\roman*}), }
\setlist[itemize]{label=\raisebox{0.25ex}{\tiny$\bullet$}}
\usepackage[backref, colorlinks, linktocpage, citecolor = \nicecolor, linkcolor = \nicecolor, urlcolor = \nicecolor]{hyperref}

\usepackage{mathtools}
 
\renewcommand{\thefootnote}{} 

\newtheorem{thm}{Theorem}[section]
\newtheorem*{thm*}{Theorem}

\newtheorem{lem}[thm]{Lemma}
\newtheorem{prop}[thm]{Proposition}
\newtheorem{defi}[thm]{Definition}

\newtheorem{rmk}[thm]{Remark}
\newtheorem{ex}[thm]{Example}

\renewcommand{\to}{\longrightarrow}
\newcommand{\rat}{\dashrightarrow}

\newcommand{\Z}{\ensuremath{\mathbb{Z}}}
\newcommand{\Q}{\ensuremath{\mathbb{Q} }}
\newcommand{\F}{\ensuremath{\mathbb{F} }}

\newcommand{\R}{\ensuremath{\mathbb{R}}}
\newcommand{\C}{\ensuremath{\mathbb{C} }}
\newcommand{\p}{\ensuremath{\mathbb{P} }}

\DeclareMathOperator{\PGL}{PGL}

\DeclareMathOperator{\Aut}{Aut} 
\DeclareMathOperator{\Bir}{Bir}
\newcommand{\Pic}{{\mathrm{Pic}}}

\DeclareMathOperator{\GL}{GL} 
 
\DeclareMathOperator{\Cr}{Cr}
\DeclareMathOperator{\Gal}{Gal}

\DeclareMathOperator{\sym}{Sym}

\usepackage{extarrows}
\usepackage{mathtools}
\usepackage{graphicx}
\usepackage{tikz-cd}
\usepackage{tikz}
\tikzset{commutative diagrams/.cd, arrow style=tikz, diagrams={>=latex}}

\begin{document}
\author{Ahmed Abouelsaad}
\pagestyle{myheadings}
\markright{Finite subgroups of maximal order of the Cremona group over the rationals} 
\title{\uppercase{Finite subgroups of maximal order of the Cremona group over the rationals}} 
 
\renewcommand{\thefootnote}{\fnsymbol{footnote}}
 
\maketitle 
\begin{abstract}
Let $\Cr_\Q(2)$ be the Cremona group of rank $2$ over rational numbers. we give a classification of large finite subgroups $G$ of $\Cr_\Q(2)$ and give a new sharp bound smaller (but not multiplicative) than $M(\Q)=120960 = 2^7\cdot3^3\cdot5\cdot7$; the one given in \cite{MR2567402}. In particular, we prove that any finite subgroup $G \subset\Cr_\Q(2)$ has order $\mid G\mid \le  432$ and Lemma \ref{lemm-17} provides a group of order $432$. We use the modern approach of minimal $G-$surfaces, given a (smooth) rational surface $S\subset\p^2$ defined over $\Q$, we study the finite subgroups $G \subset \Aut_{\Q}(S)$ of automorphisms of $S$.  We give the best bound for the order of $G\subset\Aut(S)$ for surfaces with a conic bundle structure invariant by $G$. We also give the best bound for the order of $G\subset \Aut_\Q(S)$ for all rational Del Pezzo surfaces of some given degree. In addition, we give descriptions of the finite subgroups of automorphisms of conic bundles and Del Pezzo surfaces of maximal size.
\end{abstract}
\section{Introduction}
The Cremona group $\Cr_k(n)$ over a field $k$ is the group of birational automorphisms of the projective space $\p^n_k$, or, equivalently, the group of $k-$automorphisms of the field $k(x_1, x_2, . . . , x_n)$ of rational functions in $n$ independent variables, it is also denoted by $\Bir(\p^n_k)$.\\

The group $\Cr_k(1)$ is the group of automorphisms of the projective line, and hence it is isomorphic to the projective linear group $\PGL(2,k)$. The classification of finite subgroups of $\PGL(2)$ is well-known and goes back to F. Klein. In the case $n = 2$ and $k=\C$, the group $\Cr_{\C}(2)$ is the plane Cremona group over the field of complex numbers.  The classification of finite subgroups of $\Cr_{\C}(2)$ has been studied and has a long history, going back to the $19$th century (see the references in \cite{MR2286582} and \cite{MR2641179}).\\

In this paper, we consider the plane Cremona group over the field of rational numbers, denoted by $\Cr_{\Q}(2)$. This group  $\Cr_{\Q}(2)$ has been studied in \cite{MR2567402}. The main theorem in \cite{MR2567402} gives a sharp multiplicative bound for the orders of the finite subgroups $G$ of $\Cr_\Q(2)$, which is $M(\Q)=120960 = 2^7\cdot3^3\cdot5\cdot7$, this means that the order of any finite subgroup of $\Cr_\Q(2)$ divides $2^7\cdot3^3\cdot5\cdot7$.
 
Our goal is to give a classification of finite subgroups $G$ of $\Cr_\Q(2)$ and give a new sharp bound smaller (not multiplicative) than the one given in \cite{MR2567402}. Our main theorem is the following:
\begin{thm}\label{thmm-11}
   If $G$ is a finite subgroup of $\Bir(\p^2_{\Q})$, then $\mid G\mid \le 432$. Moreover, there is a finite subgroup of $\Bir(\p^2_{\Q})$ of order $432$.
\end{thm}

As already used in \cite{MR2567402}, every finite subgroup $G\subset \Cr_{\Q}(2)$ is conjugate, by a birational map $\phi:\p^2_{\Q} \rat S$, to a finite subgroup of biregular automorphisms $\phi G \phi^{-1}\subset \Aut_{\Q}(S)$. Moreover, one can reduce to the case where S is a Del Pezzo surface or where the group preserves a conic bundle. The degree of a Del Pezzo being between $1$ and $9$ and since every finite subgroup of automorphisms of a Del Pezzo surface of degree $7$ is conjugate to a finite subgroup of $\Aut(\p^2)$, we get $9$ cases to consider. We give the best upper bound for the order of a finite group, an example, and the proof that it is maximal in each of the cases. \\

Let $S$ be a smooth projective minimal surface of degree $d\le 9$ defined over $\Q$. The following table gives the structure of a subgroup of maximal order of the automorphism group $\Aut_{\Q}(S)$ of $S$:
\begin{table}[h!]
    \centering
    \begin{tabular}{| c| c| c| c|c|}
    \hline
$G-$surface &structure of $G$& order & example & lemma\\
      \hline
 conic bundles &  $D_6\times D_6$   & $144$ &\ref{ex-05}&\ref{lemm-24}\\ 
      \hline
      $dP_9$ &$(\Z/2\Z)^2\rtimes\sym_3$ &$24$& \ref{ex-07}&\ref{lemm-26}\\  
  \hline
$dP_8$ &$(D_6\times D_6)\rtimes\Z/2\Z$ &$288$& \ref{ex-03}&\ref{lemm-27}\\  
  \hline
$dP_6$& $(\Z/6\Z)^2\rtimes D_6$    &$432$&\ref{lemm-17}&\ref{lemm-19}\\  
\hline
$dP_5$&  $ \sym_5$ &$120$& \ref{ex-09}&\ref{lemm-20}\\ 
\hline
$dP_4$ & $(\Z/2\Z)^4\rtimes \sym_3$ &$96$&\ref{ex-06}&\ref{lemm-25}\\ 
\hline
$dP_3$&$\sym_5$  &$120$&\ref{ex-02}&\ref{lemm-28}\\ 
\hline
$dP_2$ &$(\Z /2\Z )^3\rtimes\sym_3$   &$48$&\ref{ex-08}&\ref{lemm-22}\\
\hline
$dP_1$ & $D_6$   &$12$&\ref{ex-10}&\ref{lemm-23}\\
\hline
    \end{tabular}
    \label{Table:-32}
\end{table}\\
Where $dP_d$ refers to all rational Del Pezzo surfaces of degree $d$. \bigskip

\thanks{I would like to thank my PhD advisor Jérémy Blanc for suggesting the question and for interesting discussions during the preparation of this text. I extend my thanks to the Department of Mathematics and Computer Science at Basel for their support and hospitality.}
\section{Preliminaries}
We use the modern approach to the problem initiated in the works of Yuri Manin and the second author (see a survey of their results in \cite{MR1422227}). Their work gives a clear understanding of the conjugacy problem via the concept of a rational $G-$surface.\\[2mm]
 We consider groups acting on a surface. A $G-$surface, denote by $(G, S)$, is a pair in which $S$ is a (rational) surface and $G \subset \Aut(S)$ is a group acting biregularly on $S$.
  \begin{defi}
  A smooth projective surface $S$ is said to be minimal if any birational morphism from $S$ to a (nonsingular) surface is an isomorphism.
\end{defi}
 
We say that a birational map $\varphi:S\rightarrow S'$ is $G-$equivariant if the $G-$action on $S'$ induced by $\varphi$ is biregular. The birational map $\varphi$ is called a birational map of $G-$surfaces.
  \begin{defi}
      A pair $(G, S)$ is said to be minimal (or equivalently that G acts minimally on $S$) if
any $G-$equivariant birational morphism $\varphi:S\rightarrow S'$ is an isomorphism.
  \end{defi}
A birational map $\varphi:S\rat\p^2$ yields an isomorphism between $\Cr(S)$ and $\Cr(\p^2)$. The Cremona group is then isomorphic to the group of birational maps of any rational surface. We say that two groups $G\subset\Cr(S)$ and $G^{\prime}\subset\Cr(S^{\prime})$ are birationally conjugate if there exists a birational map $\varphi:S\rat S^{\prime}$ such that $\varphi G\varphi^{-1}=G^{\prime}$. The choice of $\varphi$ does not change the conjugacy class of the group and any element of the conjugacy class can be obtained in this manner. The converse is true, if $G$ is finite $($see for example a proof in \cite[Theorem $1.4$]{MR1954063}$)$. This means that the two groups represent the same conjugacy class in the Cremona group and conversely a conjugacy class of a finite subgroup $G$ of $\Cr_\C(2)$ can be realized as a birational isomorphism class of $G-$surfaces. In this way classification of conjugacy classes of subgroups of $\Cr_\C(2)$ becomes equivalent to the birational classification of $G-$surfaces. On another way, any subgroup of $\Bir(S)$ may therefore be viewed as a subgroup of the Cremona group, up to conjugacy. A $G-$equivariant analog of a minimal surface allows one to concentrate on the study of minimal $G-$surfaces, $i.e.$, surfaces which cannot be $G-$equivariantly birationally and regularly mapped to another $G-$surface. Minimal $G-$surfaces turn out to be $G-$isomorphic either to a conic bundle or a Del Pezzo surface of degree $d$ where $1\le  d \le  9$. \\
To give a smaller sharp bound, one requires to study the finite groups $G$ which may occur in a minimal $G-$surfaces defined over $\Q$ that are $G-$isomorphic to
\begin{itemize}
      \item  a conic bundle structure,
        \item a Del Pezzo surface of degree $d \le  9$.
\end{itemize}
We will study the first case in Section $3$ and the second case in Sections $4-11$.
\section{Rational Surfaces with Conic Bundle structure}
 
\begin{defi}\label{def-15}
    Let $S$ be a smooth projective rational surface defined over $ \overline{\Q}$ and $\pi : S\rightarrow \p^1$ be a morphism. We say that the pair $(S,\pi)$ is a conic bundle if a general fibre of $\pi$ is isomorphic to $\p^1$, with a finite number of exceptions: these singular fibres are the union of smooth rational curves $F_1$ and $F_2$ such that $(F_1)^2 = (F_2)^2 =-1$ and $F_1 · F_2 = 1$.
\end{defi}
Let $(S,\pi)$ and $( \tilde{S},\tilde{\pi})$ be two conic bundles. We say that $\varphi:S\dashrightarrow\tilde{S}$ is a birational map of conic bundles if $\varphi$ is a birational map which sends a general fibre of $\pi$ on a general fibre of $\tilde{\pi}$. This corresponds to ask for $\alpha\in\Aut(\p^1)$ such that $\tilde{\pi}\circ\varphi=\alpha\circ\pi$.
 
We say that a conic bundle $(S,\pi)$ is minimal if any birational morphism of conic bundles $(S,\pi)\rightarrow(\tilde{S},\tilde{\pi})$ is an isomorphism.\\
 
Let $(S, \pi)$ be some conic bundle. We denote by $\Aut(S, \pi)$ (respectively $\mathrm{Bir}(S, \pi)$) the group of automorphisms (respectively birational self-maps) of the conic bundle $(S,\pi)$, $i.e.$ automorphisms of $S$ that send a general fibre of $\pi$ on another general fibre. Observe that $\Aut(S, \pi) = \Aut(S) \cap \mathrm{Bir}(S, \pi)$.

 \begin{defi}\label{def-03}
    Let $n \in\Z$. The $n-$th Hirzebruch surface $\F_n$ is defined to be the quotient of $(\mathbb{A}^2\setminus\{0\})^2$ by the action of $(\mathbb{G}_m)^2$ given by
  \begin{eqnarray*}
        (\mathbb{G}_m)^2\times(\mathbb{A}^2\setminus\{0\})^2&\to& (\mathbb{A}^2\setminus\{0\})^2\\
        \left((\mu,\rho),(x_0,x_1,y_0,y_1)\right)&\mapsto& \left(\mu\rho^{-n}x_0,\mu x_1,\rho y_0,\rho y_1\right).
    \end{eqnarray*}
The class of $(x_0,x_1,y_0,y_1)$ will be written $[x_0:x_1;y_0:y_1]$. The projection
\begin{equation*}
    \tau_{n}:\F_n\to{\p^1},~ [x_0:x_1;y_0:y_1]\mapsto[y_0:y_1]
\end{equation*} 
identifies $\F_n$ with $\p_{{\p^1}}(\mathcal{O}_{{\p^1}}\oplus\mathcal{O}_{{\p^1}}(n)$ as a ${\p^1}-$bundle over ${\p^1}$.
\end{defi}
The disjoint sections $s_{n},s_{-n}\subset\F_n$ given by $x_0 = 0$ and $x_1 = 0$ have self-intersection $n$ and $-n$, respectively.
The fibres $f \subset\F_n$ given by $y_0 = 0$ and $y_1 = 0$ are linearly equivalent and of self-intersection $0$. We moreover get $\Pic(\F_n)=\Z f\oplus\Z s_{-n}= \Z f\oplus\Z s_{n}$.\\[3mm] We have the following result:

\begin{prop} Minimal rational surfaces $($\cite[Theorem V.10]{MR1406314}$)$ \\
Let $S$ be a minimal rational surface. Then $S$ is isomorphic to $\p^2$ or to one of the surfaces
$\F_n$, with $n\neq 1$.
\end{prop}

\begin{lem}\label{lemm-12}
    Let $S$ be a rational smooth surface defined over $\bar{\Q}$ and $\pi:S\rightarrow \p^1$ be a conic bundle. Let $G$ be a finite subgroup of $\Aut_{\bar{\Q}}(S,\pi)$, and $\bar{\pi} : G\rightarrow \Aut(\p^1_{\bar{\Q}})= \PGL (2,\bar{\Q })$ be the homomorphism given by the action on fibres. Let $G^{\prime}=\mathrm{Ker}\bar{\pi}$ and $g\in G^{\prime}\setminus\{id\}$. Then the action of the element $g$ on any smooth fiber is not trivial.
\end{lem}
\begin{proof} 
We contract orbits of components in any singular fibre, and assume that the action of $g$ to be minimal. By \cite[Lemma 3.8]{MR2486798}, any singular fibre of $\pi$ is twisted by some element of $\langle g\rangle$. If $g$ twists a singular fibre, then ($g$ is an involution) and the set $\mathrm{Fix}(g)$ of fixed locus of $g$ is a $2:1$ cover of $\p^1$ via $\pi$ (\cite[Lemma $6.1$]{MR2486798}). So no curve contained in a fiber is fixed point-wise by $g$. Otherwise, there is no singular fiber and by \cite[Lemma 3.2.]{MR2486798}, the surface $S$ is isomorphic to the Hirzebruch surface $\F _n$, for some integer $n \ge 0$. If $n=0$, then the action is given by \[[x_0:x_1;y_0:y_1]\mapsto [a\cdot x_0+b\cdot x_1:c\cdot x_0+d\cdot x_1;y_0:y_1],\]
where $\left(\begin{array}{cc} a&b\\c&d \end{array}\right)\in\PGL _2(\bar{\Q })$, since the action is the same on each fibre, then no fibre is fixed point-wise by $g$. If $n\ge1$, the the action is given by \[[x_0:x_1;y_0:y_1]\mapsto [x_0:\lambda \cdot x_1 +P_n(y_0,y_1)\cdot x_0;y_0:y_1],\] where $P_n(y_0,y_1)$ is a homogeneous polynomial of degree $n$, and $\lambda\in\bar{\Q }^*$. If $\lambda=1$, then the action is not of finite order. If $\lambda\neq1$, then the action is never the identity on any fibre. 
\end{proof}

\begin{lem}\label{lemm-24}
 Let $S$ be a rational surface defined over $\Q $ and $\pi:S\rightarrow \p^1$ be a conic bundle. If $G$ is a finite subgroup of $\Aut_{\Q }(S,\pi)$, then $\mid G\mid $ divides $144$. 
\end{lem}
\begin{proof}
    Let $S$ be a rational surface defined over $\Q $ and $\pi:S\rightarrow \p^1$ be a conic bundle. Let $G$ be a finite subgroup of $\Aut_{\Q }(S,\pi)$. We have a natural homomorphism $\bar{\pi} : G\rightarrow \Aut(\p^1_{\Q }) = \PGL (2,\Q )$ that satisfies $\bar{\pi}(g)\pi = \pi g$, for every $g \in G$. Let $G^{\prime} = \mathrm{ker}\bar{\pi} $, hence, we have the following exact sequence
    \begin{equation}
        1\rightarrow G^{\prime}\rightarrow G\rightarrow\bar{\pi}(G)\subseteq\Aut(\p^1_{\Q })\rightarrow1.
    \end{equation}
Hence, $\mid \bar{\pi}(G)\mid \subseteq\Aut(\p^1_{\Q }) = \PGL (2,\Q )$ and by \cite[Theorem 4.2.]{MR2681719}, $\mid \bar{\pi}(G)\mid $ divides $12$. After extending the scalar to $\bar{\Q}$, we apply Lemma \ref{lemm-12}, then any element $g\in G^{\prime}\setminus\{id\}$ acts non-trivially on any smooth fibre, this implies  that $G^{\prime}\subseteq\Aut(\p^1_{\Q})=\PGL (2,\Q)$. By \cite[Theorem 4.2.]{MR2681719}, $\mid G^{\prime}\mid $ divides $12$. Since $\mid G\mid =\mid G^{\prime}\mid \cdot\mid \bar{\pi}(G)\mid $, then $\mid G\mid $  divides $12\cdot12=144$.
\end{proof}

\begin{ex}   \label{ex-05}  
Let $S=\p^1\times\p^1$ and $\pi:S\rightarrow \p^1$ be the projection on a factor, that is a conic bundle. Any subgroup of  $\Aut_{\Q}(S,\pi)\subset \Aut(\p^1\times\p^1)$ is isomorphic to a subgroup $H$ of $\PGL_2(\Q)\times\PGL_2(\Q)$. Let $G\subset\PGL_2(\Q)$  be a finite subgroup, then the group $H:=G\times G$ is a finite subgroup of $\PGL_2(\Q)\times\PGL_2(\Q)$, hence, the subgroup of $\Aut(\p^1\times\p^1)$ isomorphic to $H$ preserves the fibration through $\pi$. Moreover, if  $G\cong D_6$ (a dihedral group of order $12$), then $H$ is of order $144$.  
\end{ex}
\section{\texorpdfstring{Del Pezzo surfaces of degree $6$}{Del Pezzo surfaces of degree 6}}\label{section:6}
 
Let $S$ be a Del Pezzo surface of degree $6$, defined over $\Q $. Let $X$ be the blow-up of $\p^2$ at the three points $[1:0:0],~[0:1:0]$ and $[0:0:1]$, therefore, $X$ is the standard Del Pezzo surface given in $\p^2 \times \p^2$ given by
\[X=\left\{[x_0:x_1:x_2][y_0:y_1:y_2]\in \p^2\times \p^2\mid ~x_0y_0=x_1y_1=x_2y_2\right\}.\]
Over the algebraic closure of $\Q $, we have an isomorphism $S_{\overline{\Q }}\rightarrow X_{\overline{\Q }}$. Let $L/\Q $ be the smallest Galois field extension such that the six $(-1)-$curves are defined over $L$. The surface $S$ has six $(-1)-$curves defined over $L$ but not necessarily over $\Q $. The conjugation of the action gives an action of $\Gal(L/\Q )$ on $S_{L}$. Indeed, if we write the isomorphism $\varphi: S_L\rightarrow X_L$ and let $g\in \Gal(L/\Q )$, then the following diagram 
\begin{equation*}
   \begin{array}{ccc}
     \hspace{5mm} S_L &\stackrel{\varphi }{\longrightarrow}&X_L \\
   g \downarrow& &  \downarrow \Tilde{g} \\
    \hspace{5mm} S_L & \stackrel{\varphi }{\longrightarrow}&X_L
    \end{array}
\end{equation*}
commutes, where $\Tilde{g}=\varphi g\varphi^{-1}=g \tau$ for some $\tau\in \Aut(X_L)$. The Galois group $\Gal(L/\Q )$ acts on the hexagon by symmetries, so we have a group homomorphism
\[\Gal(L/\Q ) \stackrel{\varrho }{\longrightarrow} \sym_3 \times \Z /2 \Z .\]
This map $\varrho$ is injective as if we let $L^{\prime}=L^{\mathrm{Ker}(\varrho)},~\mathrm{Ker}(\varrho)$ acts trivially on the hexagon, so all $(-1)-$curves are fixed by the kernel, hence they are all defined over $L^{\prime}$ which gives $L^{\prime}=L$, and then $\mathrm{Ker}(\varrho)=\Gal(L/L)$. 
 
The group $\Aut(S_{L})$ (resp. $\Aut(X_L)$) acts by symmetries on the hexagon of $S_{L}$ (resp.$X_{L}$), which induces homomorphisms
\[\Aut(S_{L})\stackrel{\rho }{\longrightarrow} \sym_3 \times \Z /2\Z ,~~~~~~~~~\Aut(X_{L})\stackrel{\hat{\rho} }{\longrightarrow} \sym_3 \times \Z /2\Z .\]
 
We have the following splitting exact sequence 
 
\begin{equation}\label{Aut:X_L:exact}
1 \stackrel{~ }{\rightarrow}\left(L^{*}\right)^2\stackrel{~ }{\rightarrow} \Aut(X_L) \xrightarrow[\mathrm{the~hexagon}]{\mathrm{action~on}} \sym_3\times \Z /2\Z \stackrel{~ }{\rightarrow}1,
\end{equation}
 
where $\Aut(X_L)=\left(L^{*} \right)^2\rtimes(\sym_3 \times \Z /2 \Z )=\left(L^{*}\right)^2\rtimes D_6$, where $D_6$ is the dihedral group of order $12$.  We are looking for a finite subgroup $G$ of $\Aut(S_{\Q })$ where
\begin{equation*}
    \Aut(S_{\Q })= \{\psi\in \Aut(S_{L})\mid  ~\psi g=g\psi ~~\forall g\in \Gal(L/\Q )\}.
\end{equation*}
So we obtain by conjugation a finite subgroup of $\Aut(X_L)$ that can be described as follows
\[H\subseteq \varphi \Aut(S_{\Q })\varphi^{-1}\subseteq \Aut(X_{L})=\left(L^{*}\right)^2\rtimes D_6,\]
\begin{equation*}
H\subseteq \varphi \Aut(S_{\Q })\varphi^{-1}=\left\{\alpha\in \Aut(X_{L})\mid ~\alpha \Tilde{g}= \Tilde{g} \alpha\right\},
\end{equation*}
so we have the following exact sequence
\begin{equation}\label{lemm-13}
1 \stackrel{~ }{\rightarrow} H\cap\left(L^{*}\right)^2\stackrel{~ }{\rightarrow} H\xrightarrow[\mathrm{the~hexagon}]{\mathrm{action~on}} \sym_3\times \Z /2\Z .
\end{equation}
We define $\Delta_1$ to be the image of $H$ in $D_6$, as $\Gal(L/\Q )$ acts on the hexagon faithfully and $\Gal(L/\Q )\subseteq D_6$, so every element in $\Delta_1$ should commute with the image $\Tilde{g}$. Since $H \cap (L^*)^2$ is a finite subgroup of  $(L^*)^2$, it contains only roots of unity. We define $n$ to be the largest order of a root of unity contained in $H \cap (L^*)^2$. So we have $\Q [\zeta_n]$ as an intermediate field of the Galois field extension $L/\Q $, where $\zeta_n$ is the primitive $nth$ root of unity. By the fundamental theorem of Galois theory, $L/\Q [\zeta_n]$ is Galois, hence we have a surjective group homomorphism $\Gal(L/\Q )\rightarrow \Gal(\Q [\zeta_n]/\Q )$, $g\mapsto g\mid _{\Q [\zeta_n]}$, this leads to the existence of the following the exact sequence 
\begin{equation}\label{Gal:exact}
    1\rightarrow \Gal(L/\Q [\zeta_n])\rightarrow \Gal(L/\Q )\rightarrow \Gal(\Q [\zeta_n]/\Q )\cong \Z^*_{n}\rightarrow 1.\end{equation}

\begin{lem}\label{lemm-17}
Let $S$ be a Del Pezzo surface of degree $6$ over $\Q $ and assume that we have an isomorphism with $X$ defined over $L=\Q [\zeta_6]$, such that the action of the Galois group is given by $g:\zeta_6\mapsto \zeta_6^5$, where $g^2=id$. One can takes  \begin{equation*}
           \Tilde{g}:[x_0:x_1:x_2][y_0:y_1:y_2]\mapsto [g(y_0):g(y_1):g(y_2)][g(x_0):g(x_1):g(x_2)]
\end{equation*}
where $\Tilde{g}^2=id$, then $\Aut_{\Q }(S)$ contains a finite subgroup of order $432$ and we can describe the generators explicitly in $\Aut(X)$. Moreover, the surface $S$ is rational.
\end{lem} 
\begin{proof}
    We have $\Gal(L/\Q )=\Gal(\Q [\zeta_6]/\Q )=\langle g\rangle $. So the group $\Aut_{\Q}(S)$ contains a finite subgroup $H=\langle  h_1,h_2,h_3,\alpha_1,\ldots,\alpha_6\rangle $ that is generated by
\begin{align*}
    h_1:&[x_0:x_1:x_2][y_0:y_1:y_2]\mapsto [x_1:x_2:x_0][y_1:y_2:y_0], \\
     h_2:&[x_0:x_1:x_2][y_0:y_1:y_2]\mapsto [x_0:x_2:x_1][y_0:y_2:y_1], \\
      h_3:&(X,Y)\mapsto (Y,X), \\
    \alpha^{i}_{j}:&[x_0:x_1:x_2][y_0:y_1:y_2]\mapsto [x_0:\zeta_6^{i}x_1:\zeta_6^{j}x_2][y_0:\zeta_6^{6-i} y_1:\zeta_6^{6-j} y_2],
\end{align*}
where $0\le  i\le  5$ and $1\le  j\le  6$. Then, $H\cong(\Z/6\Z)^2\rtimes D_6$ and thus has $432$ elements. To prove that the surface is rational, the point $P=[1:1:1][1:1:1]$ is fixed by $\Tilde{g}$. Let $X_L$ be the hexagon over $L$ which is the blow-up of the three points $p_1=[1:0:0],~p_2=[0:1:0]$ and $p_3=[0:0:1]$ in $\p^2$. Since we have the point $p_4=[1:1:1]$ in the middle defined over $L$, then we blow up this point. We will get $X^{\prime}_L$ as a Del Pezzo surface of degree $5$, the blow up of $S$ at a $\Q-$rational point is then a Del Pezzo surface of degree $5$ with a $\Q-$rational point and is thus rational \cite[Theorem 3.15]{MR225780}.
\end{proof}

\begin{lem}\label{lemm-19}
   Let $S$ be a Del Pezzo surface of degree $6$ defined over $\Q $.  If $G\subset\Aut_{\Q}(S)$ is finite subgroup, then $\mid G\mid \le 432$.
\end{lem}
\begin{proof} Let $L/\Q $ be the smallest finite Galois extension such that all $(-1)-$curves of $S$ are defined over $L$ and $\Q[\zeta_n]$ be an intermediate field, where $n$ is the largest order of a root of unity contained in $L^*$ and $\zeta_n$ is a primitive $nth$ root of unity. Let $g\in\Gal(L/\Q)$ and $\Tilde{g}$ be the corresponding Galois action. In order to calculate the elements in $H\cap\left(L^{*}\right)^2$, we take $\lambda, \mu \in L^*$ and $\tau\in H\cap\left(L^{*}\right)^2 $, then one can assume that
\begin{equation}\tau:[x_0:x_1:x_2][y_0:y_1:y_2]\mapsto [x_0:\frac{1}{\lambda}\cdot 
 x_1:\frac{1}{\mu}\cdot  x_2][y_0:\lambda\cdot y_1:\mu \cdot y_2],\end{equation}
and $\tau$ commutes with $\Tilde{g},~ i.e. ~\tau=\Tilde{g}^{-1}\tau \Tilde{g}$.
Let $\Gal(L/\Q )$ be isomorphic to a subgroup of $D_6$. We apply the exact sequence \eqref{Gal:exact}, the group ($\Z_n)^*$ is a quotient of a subgroup of $D_6$. The extension $\Q [\zeta_n]/\Q$ is abelian and the group $\Gal(\Q [\zeta_n]/\Q)$ is quotient of $\Gal(L/\Q)$, $\Z_n^*$  is abelian. Therefore, it has order at most $m\le 6$; as $D_6$ is not abelian. Hence, $m=\phi( n )$, which implies that $n$ is at most $18$, where Euler$^{\prime}$s totient function $\phi( n )$.

Let $D_6= \sym_3 \times \Z_2$, where $\sym_3$ permutes the coordinates $x_i$ and $y_i$ in the same way, i.e., it is corresponding to 
\[[x_1:x_2:x_3][y_1:y_2:y_3]\mapsto [ x_{\delta(1)}: x_{\delta(2)} : x_{\delta(3)} ][ y_{\delta(1)}:y_{\delta(2)}:y_{\delta(3)}], \]
for some $\delta\in\sym_3$ and $\Z/2$ is an exchange, i.e., it is corresponding to $(X,Y)\mapsto (Y,X)$. One can study the cases of the Galois group $\Gal(L/\Q)$. 
\begin{itemize}
  \item [a)] If $\Gal(L/\Q)$ contains an element of order $2$ that is a transposition (resp. product of a transposition and an exchange), it corresponds to 
\[\alpha:[x_{0}:x_{1}:x_{2}][y_{0}:y_{1}:y_{2}]\mapsto [\frac{1}{\kappa}g(y_{0}):\frac{1}{\nu}g(y_{2}):g(y_{1})][\kappa\cdot g(x_{0}):\nu \cdot g(x_{2}):g(x_{1})]\]
(respectively, 
 \[\beta:[x_{0}:x_{1}:x_{2}][y_{0}:y_{1}:y_{2}]\mapsto [\frac{1}{\kappa}g(x_{0}):\frac{1}{\nu}g(x_{2}):g(x_{1})][\kappa\cdot g(y_{0}):\nu \cdot g(y_{2}):g(y_{1})]),\]
for some $\kappa,\nu\in L^*$. Hence, $\alpha^{-1}\tau \alpha$ is then given by 
\[[x_{0}:x_{1}:x_{2}][y_{0}:y_{1}:y_{2}]\mapsto [ x_{0}: g^{-1}(\mu)x_{1} : g^{-1}(\lambda)x_{2}][ y_{0}:\frac{1}{g^{-1}(\mu)}y_{1}:\frac{1}{g^{-1}(\lambda)}y_{2}] \]
(respectively, $\beta^{-1}\tau \beta$ is then given by 
\[[x_{0}:x_{1}:x_{2}][y_{0}:y_{1}:y_{2}]\mapsto [ x_{0}: \frac{1}{g^{-1}(\mu)}x_{1} : \frac{1}{g^{-1}(\lambda)}x_{2}][ y_{0}:g^{-1}(\mu)y_{1}:g^{-1}(\lambda)y_{2}]),\]
The equation $\tau=\alpha^{-1}\tau \alpha$ (resp. $\tau=\beta^{-1}\tau \beta$) holds provided that $\mu$ and $\lambda$ satisfy the two conditions $\mu\cdot g^{-1}(\lambda)=1$ and $\lambda\cdot g^{-1}(\mu)=1$ (resp. $\mu- g^{-1}(\lambda)=0$ and $\lambda- g^{-1}(\mu)=0$). One can see that one of $\lambda$ and $\mu$ always depends on the other, this implies to that we have at most $n$ possibilities for $\lambda$ and $\mu$, then $\mid H\cap\left(L^{*}\right)^2\mid \le  18$. Moreover, the image $\Delta_1$ has at most $12$ elements. So, we conclude that $\mid H\mid =\mid H\cap\left(L^{*}\right)^2\mid \cdot \mid \Delta_1\mid \le  18\cdot 12=216$.\\

\item [b)] If $\Gal(L/\Q)$ contains an element that is a $3-$cycle, it corresponds to
  \[\alpha:[x_{0}:x_{1}:x_{2}][y_{0}:y_{1}:y_{2}]\mapsto [\frac{1}{\kappa}g(x_{1}):\frac{1}{\nu}g(x_{2}):g(x_{0})][\kappa\cdot g(y_{1}):\nu \cdot g(y_{2}):g(y_{0})],\]
for some $\kappa,\nu\in L^*$. Hence, $\alpha^{-1}\tau \alpha$ is then given by 
\[[x_{0}:x_{1}:x_{2}][y_{0}:y_{1}:y_{2}]\mapsto [ x_{0}: g^{-1}(\mu)x_{1} : \frac{g^{-1}(\mu)}{g^{-1}(\lambda)}x_{2}][ y_{0}:\frac{1}{g^{-1}(\mu)}y_{1}:\frac{g^{-1}(\lambda)}{g^{-1}(\mu)}y_{2}].\]
The equation $\tau=\alpha^{-1}\tau \alpha$ holds provided that $\mu$ and $\lambda$ satisfy the three conditions $\lambda\cdot g^{-1}(\mu)=1$ and $g^{-1}(\lambda)-\mu\cdot g^{-1}(\mu)=0$. The first condition gives $\lambda=\frac{1}{g^{-1}(\mu)}$, $i.e.$, $\lambda$ depends on $\mu$. We have at most $n$ possibilities for $\mu$, then $\mid H\cap\left(L^{*}\right)^2\mid \le  18$. So, we conclude that $\mid H\mid =\mid H\cap\left(L^{*}\right)^2\mid \cdot \mid \Delta_1\mid \le  18\cdot 12=216$.
\end{itemize}
The remaining cases now do not have any transposition nor any $3-$cycle. It only remains the cases where the group $\Gal(L/\Q)$ is trivial or of order $2$. If it is trivial, then $n\le 2$, so $\mid H\cap\left(L^{*}\right)^2\mid \le  2^2$. Hence, $\mid H\mid =\mid H\cap\left(L^{*}\right)^2\mid \cdot \mid \Delta_1\mid \le  2^2\cdot 12=48$. If $\Gal(L/\Q)$ is of order $2$, then $n\le 6$, so $\mid H\cap\left(L^{*}\right)^2\mid \le  6^2$. Hence, $\mid H\mid =\mid H\cap\left(L^{*}\right)^2\mid \cdot \mid \Delta_1\mid \le  6^2\cdot 12=432$.
\end{proof}

\section{\texorpdfstring{Del Pezzo surfaces of degree $5$}{Del Pezzo surfaces of degree 5}}

A complex Del Pezzo surface of degree $5$ is given by the blow-up of $4$ points of $\p^2_{\C}$, in general position. The following result is classical.
\begin{lem}\label{lemm-21}
 Let $k$ be any field and let $S$ be the blow-up of four points in general position, then the automorphism group $\Aut(S_k)$ is always $\sym_5$.
\end{lem}
\begin{proof}
  Let $S$ be the blow-up of the four points $A_1,..,A_4$ of $\p^2(k)$, in general position. Let $\pi:S \rightarrow\p^2$ be the blow-down, then the exceptional divisors on the surface $S$ are given by:
\begin{itemize}
\item $E_i=\pi^{-1}(A_i)$, the $4$ pull-backs of the points $A_i$ where $i=1,..,4$.
\item $l_{ij} = \pi^{-1}(A_iA_j)$ for $1\le  i,j \le  4,~i \neq j$, the $6$ strict pull-backs of the lines through $2$ of the $A_i’s$. 
\end{itemize}
Furthermore, each divisor intersects three others. By \cite[Proposition $5.1$]{MR2486798}, There are thus $5$ sets of $4$ skew exceptional divisors on the surface $S$, namely  $\{E_1,l_{23},l_{24},l_{34}\}$, $\{E_2,l_{13},l_{14},l_{34}\}$, $\{E_3,l_{12},l_{14},l_{24}\}$, $\{E_4,l_{23},l_{13},l_{23}\}$, $\{E_1, E_2, E_3,E_4\}$, and the action on these sets gives rise to an isomorphism $\rho: \Aut(S_k) \rightarrow \sym_5$.
\end{proof}

 \begin{ex}\label{ex-09}
    Let $S$ be the Del Pezzo surface of degree $5$ defined over $\Q$ that is the blow-up of the four points $A_1,..,A_4$ of $\p^2(\Q)$, in general position. Since under the action of linear automorphisms, all sets of four non-collinear points of $\p^2$ are equivalent, we may assume that the four points $A_1 =(1:0:0),~A_2 =(0:1:0),~A_3 =(0:0:1)$ and $A_4 =(1:1:1)$ and they are all defined over $\Q$. Furthermore, by Lemma \ref{lemm-21}, $\Aut(S_{\Q})\cong\sym_5$.
\end{ex}

\begin{lem}\label{lemm-20}
  Let $S$ be a Del Pezzo surface of degree $5$ defined over $\Q $. If $G\subseteq\Aut(S_{\Q})$ is a subgroup, then $\mid G\mid \le 120$. 
\end{lem}
\begin{proof}
   Let $S$ be a Del Pezzo surface of degree $5$ defined over $\Q$, and let $G\subseteq\Aut(S_{\Q})$. Since $\Aut(S_{\Q})\subseteq\Aut(S_{\C})$. By Lemma \ref{lemm-21}, $\Aut(S_{\C})$ has order $120$, so $\mid G\mid \le 120$. 
\end{proof}
 \section{\texorpdfstring{Finite subgroups of $\PGL_3(\Q)$}{Finite subgroups of PGL3}}
 \begin{thm}  \cite[Minkowski$^{\prime}$s theorem]{MR1580123}\label{thmm-05}, \cite[Theorem 1]{MR2336645}\\
Let $n\ge   1$ be an integer, and let $p$ be a prime number. $v_p(x)$ is the $p-$adic valuation of a rational number $x$, and $v_p(H)$ is the order of an $p-$Sylow of $\Z $. Define:
\[M(n,p)=[\frac{n}{p-1}]+[\frac{n}{p(p-1)}]+[\frac{n}{p^2(p-1)}]+\ldots \]
Then:
\begin{itemize}
    \item [(i)] If $H$ is a finite subgroup of $\mathrm{GL}_n(\Q )$, we have $v_p(H)\le  M(n,p)$.
    \item [(ii)] There exists a finite $p-$subgroup $H$ of $\mathrm{GL}_n(\Q )$ with $v_p(H) = M(n,p)$.
\end{itemize}
\end{thm}
Define \[M(n)=\prod_{p} p^{M(n,p)}.\]
Part $(i)$ of Theorem \ref{thmm-05} says that the order of any finite subgroup of $\mathrm{GL}_n(\Q )$ divides $M(n)$, and part $(ii)$ says that $M(n)$ is the smallest integer having this property. Hence $M(n)$ is a sharp multiplicative bound for $\mid H\mid $.\\
\begin{rmk}\label{remk-03}
    Since $M(3,2)=4,~M(3,3)=1$ and $M(3,p)=0$ for all primes $p>3$, we have 
    \[M(3)=2^4\cdot 3^1\prod_{p> 3} p^0=48,\] then the order of any finite subgroup $G$ in $\mathrm{GL}_3(\Q )$ divides $48$. Moreover, for each group $H\subset\mathrm{SL}_3(\Q )$, we have a group $G=\{\pm h\mid \in H\}\subset \mathrm{GL}_3(\Q )$. But $\mid G\mid =2\cdot\mid H\mid $, which means that the order of any finite subgroup $H$ in $\mathrm{SL}_3(\Q )$ has to divide $24$.
\end{rmk}
\begin{lem}\label{lemm-18}
    Let $K$ be a field, let $a,b \in K^*$ be elements, and let $m,n$ be coprime positive integers. If $a^m=b^n$, there exists $c \in K$ such that $a=c^n$ and $b=c^m$.
\end{lem}
\begin{proof}
    As $m,n$ are coprime, there exists $x,y$ in $\Z $ such that $nx+my=1$. We then define $c=a^xb^y$ and get
\[c^n=a^{nx}b^{ny}=a^{nx+my}=a,\]
\[c^m=a^{mx}b^{my}=b^{nx+my}=b\]
\end{proof}
\begin{lem}\cite[Lemma 3.5]{hu2022jordan}\label{lemma:hu}\\
    Let $K$ be a field, $ n\ge  2$ be an integer, and $r$ be an integer coprime to $n$. If $\alpha$ in $\PGL _n(K)$ is of order $r$, then there exists an element $A$ in $\mathrm{SL}_n(K)$ of order $r$ such that $\alpha$ is the class of $A$.
\end{lem}
\begin{proof}
Let $\alpha$ in $\PGL _n(K)$ be an element of order $r$, then there exists $B$ in $\mathrm{GL}_n(K)$ whose class in $\PGL _n(K)$ is $\alpha$ such that $B^r=\lambda\cdot Id$, where $Id$ is the identity matrix and $\lambda\in K^*$. The equation $B^r= \lambda\cdot Id$ implies, taking the determinant, that $\mathrm{det}(B)^r=\lambda^n$. This implies, by Lemma \ref{lemm-18}, that $\mathrm{det}(B)=\mu^n$ and $\lambda=\mu^r$ for some $\mu$ in $K^*$. Replacing $B$ with $B/\mu$ gives an element of $\mathrm{SL}_n(K)$ that has order $r$, and whose class is $\alpha$.
\end{proof} 
\begin{rmk}\label{remk-06} \cite[p.~124]{MR1176100}\\ Any finite subgroup of $\mathrm{GL}_n(\Q )$ is conjugate to a finite subgroup of $\mathrm{GL}_n(\Z )$.
\end{rmk}
\begin{lem}\label{element:O3P2}
    Let $\alpha\in\PGL_2(\Q)$ be an element of order $3$, then $\alpha$ or $\alpha^{-1}$ is conjugate to $\left(\begin{matrix}
       0&-1\\1&-1 
    \end{matrix}\right)$.
\end{lem}
\begin{proof}
   Let $\alpha\in\PGL _2(\Q )$ be an element of order $3$, By Lemma \ref{lemma:hu}, there exists an element $A$ in $\mathrm{SL}_2(\Q )$ of order $3$ whose class is $\alpha$. By Remark \ref{remk-06}, $A$ is conjugate to an element $B$ in $\mathrm{SL}_2(\Z )$ of whose class $\beta$ has order $3$. The quotient of the group $\mathrm{SL}_2(\Z )$ by its center $\{\pm id\}$ is a free product of $\Z /2 \Z $ and $\Z /3\Z $ generated by the classes $[S]$ of $S$ and $[T]$ of $T$, where 
   \[S=\left(\begin{array}{cc}
       0& -1 \\
       1 & -1
   \end{array}\right),~ T=\left(\begin{array}{cc}
       0& 1 \\
       -1 &0
   \end{array}\right),\]
   see \cite[Chap. 8]{MR340283}. Since the class of $S$ has order $3$, then $A$ or $A^{-1}$ is conjugate to the matrix $S$.
\end{proof}
\begin{lem}\label{Lemma:pgroup}
   Let $p$ be a prime number. If $\alpha$ is an element of order $p$ in $\PGL _3(\Q )$, then $p\le  3$. 
\end{lem}
\begin{proof}
   We prove that $p\le  3$ by contradiction. Let $p$ be a prime number $> 3$. Let $\alpha\in G$ be an element of order $p$. Since $(p,3)=1$, from Lemma \ref{lemma:hu}, there exists $A$ in $\mathrm{SL}_3(\Q )$ of order $p$ where $\alpha$ is the class of $A$. This implies that the group $H=\langle A\rangle$ in $\mathrm{SL}_3(\Q )$ has size $p$. By Remark \ref{remk-03}, $p$ divides $24$, so $p\in\{2,3\}$ which is a contradiction.
\end{proof}
 
\begin{rmk}\label{remk-04}
    If $H$ is a $2-$group in $\PGL _3(\Q )$, then by Remark \ref{remk-03} and Lemma \ref{lemma:hu}, we have $\mid H\mid \le  2^3$.
\end{rmk}
\begin{lem}\label{lemm-29}
Every element $\alpha$ of order $3$ of $\PGL _3(\Q )$ is conjugate to \[B=\left(\begin{matrix}
       0&0  & \lambda \\
          1&0  &0\\
           0&1 &0
    \end{matrix}\right),\]
    for some $\lambda\in\Q \setminus\{0\}$.
\end{lem}
\begin{proof}
    Let $A\in\mathrm{GL}_3(\Q )$ be an element whose class is $\alpha$ in $\PGL _3(\Q )$ has order $3$, then we have $A^3=\lambda\cdot Id$, where $Id$ is the identity matrix and $\lambda\in\Q \setminus\{0\}$. Let $v\in\Q ^3\setminus\{0\}$ be a nonzero vector.
  \begin{itemize}
        \item [(a)] If $v, Av$ and $A^2v$ are three linearly independent vectors, because $A$ sends $v$ to $A\cdot v,~A\cdot v$ to $A^2\cdot v$ and $A^2\cdot v$ to $A^3\cdot v=\lambda \cdot Id\cdot v=\lambda \cdot v$, we find that $A$ is conjugate to the matrix $B$.
        \item [(b)] If $v, Av$, and $A^2v$ are linearly dependent vectors, then we have two cases:
\end{itemize}
        \begin{itemize}
            \item [(b.1)] In case where $v, Av$ are linearly dependent vectors, then $Av=\mu v$, where $\mu \in \Q \setminus\{0\}$. Replacing $A$ by $A/\mu$,we can assume that $\mu=1$. By a coordinate change, $v=\left(\begin{array}{c}
                 1\\
                 0\\0 
            \end{array}\right)$, hence
\[A=\left(
  \begin{array}{ccc}
\begin{matrix}
    1
\end{matrix}
&
\begin{matrix}
    a_{1} & a_{2}
\end{matrix}\\
\begin{matrix}
    0\\0
\end{matrix}&
\begin{matrix}
    C
\end{matrix}
\end{array}
\right),
\]
where $a_i\in\Q ,$ for all $i=1,2$ and $C$ in $\mathrm{GL}_2(\Q )$. With this choice of $A$,

\[A^3=\left(
  \begin{array}{ccc}
\begin{matrix}
    1
\end{matrix}
&
\begin{matrix}
    u & v
\end{matrix}\\
\begin{matrix}
    0\\0
\end{matrix}&
\begin{matrix}
    C^3
\end{matrix}
\end{array}
\right)=Id,
\]
implies to $u=0,~v=0$ and the class of $C$ in $\PGL _2(\Q )$ is of order $3$, then by Lemma \ref{element:O3P2}, the class of $C$ is conjugate to the matrix $\left(\begin{matrix}
       0&-1\\1&-1 
    \end{matrix}\right)$ or its inverse $\left(\begin{array}{cc}
-1 & 1 
\\
 -1 & 0 
\end{array}\right)$. We may then replace $C$ by a conjugate and assume that $C= \mu\cdot\left(\begin{matrix}
       0&-1\\1&-1 
    \end{matrix}\right)$ or $C=\mu\cdot\left(\begin{matrix}
       0&-1\\1&-1 
    \end{matrix}\right)$ ), where $\mu\in\Q^*$. As $A^3=Id$, we obtain $\mu=1$. Hence 
    
    \begin{itemize}
        \item [(b.1.1)] If we assume that
$A=\left(
  \begin{array}{ccc}
\begin{matrix}
    1
\end{matrix}
&
\begin{matrix}
    a_{1} & a_{2}
\end{matrix}\\
\begin{matrix}
    0\\0
\end{matrix}&
\begin{matrix}
    0&-1\\1&-1
\end{matrix}
\end{array}
\right),
$ then $A^3=Id$ for all $a_{1},a_{2}\in\Q $. Moreover, we have $M^{-1}AM=B$, where 
\[M=\left(\begin{array}{ccc}
a_{1}+2 a_{2}-1 & -2 a_{1}-a_{2}-1 & a_{1}-a_{2}-1 
\\
 -3 & 3 & 0 
\\
 -3 & 0 & 3 
\end{array}
\right).\]
 
\item [(b.1.2)] If we assume that
$A=\left(\begin{array}{ccc}
1 & a_{1} & a_{2} 
\\
 0 & -1 & 1 
\\
 0 & -1 & 0 
\end{array}\right)
,$ then $A^3=Id$ for all $a_{1},a_{2}\in\Q $. Moreover, we have $M^{-1}AM=B$, where $\lambda=1$ and 
\[M=\left(\begin{array}{ccc}
a_{1}^{2}+a_{1} a_{2}+a_{2}^{2}+1 & -a_{1}^{2}-a_{1} a_{2}-a_{2}^{2}+1 & 1 
\\
 -2 a_{1}-a_{2} & a_{1}-a_{2} & a_{1}+2 a_{2} 
\\
 -a_{1}-2 a_{2} & 2 a_{1}+a_{2} & -a_{1}+a_{2}
\end{array}\right).\]

    \end{itemize}
   \item [(b.2)] In the case where $v, Av$ are linearly independent vectors but the three vectors are not linearly independent, up to the change of coordinates, we can assume that the two vectors $v$ and $Av$ are 
   \[\left(\begin{array}{c}
                 1\\
                 0\\0 
            \end{array}\right)~\mathrm{and}~\left(\begin{array}{c}
                 0\\
                 1\\0 
            \end{array}\right),\] and since $A$ sends the first one to the second one, and $A^2\cdot v$ is linearly dependent of $v$ and $A\cdot v$. We obtain 
\[
A=\left(
  \begin{array}{ccc}
0 & a_{1} & a_{2} 
\\
 1 & a_{3} & a_{4} 
\\
 0 & 0 & a_{5} 
\end{array}\right).
\]
where $a_i\in\Q ,$ for all $i\in\{1,..,5\}$. one can see that $a_5$ is an eigenvalue of the matrix $A$, and then we can calculate its eigenvector which is $\left(\begin{array}{c}
                 u\\
                 v\\1 
            \end{array}\right)$ for some $u,v\in\Q $. This reduces this case to the previous studied case.
\end{itemize}
\end{proof}
 
\begin{lem}\label{element3}
    Let $\alpha$ be an element of order $3^l$ of $\PGL_r(\Q )$ where $r\le 5$, then $l\le 1$.
\end{lem}
 
\begin{proof}  We prove that there is no element of order $9$, and thus there is no element of order $3^l$, where $l> 1$. Let $r\le 5$. Let $\alpha$ be an element of $\PGL _r(\Q )$ of order $9$. There is thus an element $A$ in $\mathrm{GL}_r(\Q )$ whose class is $\alpha$ and  $a\in \Q $ such that $A^9= a \cdot Id$. If we multiply $A$ by a suitable integer number that makes the entries of $A$ integers, we can assume that $a\in \Z $. The polynomial $X^9-a$ is then zero when we replace $X$ with $A$ and is thus a multiple of the minimal polynomial of $A$, which is of degree at most $5$. The minimal polynomial of $A$ has then to divide the polynomial $X^9-a$. If $a$ is not a cube in $\Z $, then $X^9-a$ is irreducible over $\Z $ \cite[Theorem 9.1]{MR1878556}. So we can we write $a=b^3$ where $b$ in $\Z \setminus\{0\}$. This implies that $(X^9-b^3)=(X^3-b)(x^6+bx^3+b^2)$. We observe that the minimal polynomial cannot divide $X^3-b$. Otherwise, we would have $A^3=b\cdot Id$, which implies that $\alpha$ is of order $3$. We now look at the polynomial $x^6+bx^3+b^2$. If this one is irreducible, we find a contradiction. In  \cite[Theorem 3.1]{MR3657411}, we find that this latter is reducible only when $b^2=n^3$ and $b=m^3-3mn$ for some $m,n$ in $\Q $. We then write $s=b/n$ in $\Q \setminus\{0\}$, and obtain $s^3=b^3/n^3=b^3/b^2=b$. We then find $a=b^3=s^9$. The matrix $A
^{\prime}=A/s$ is then such that $(A^{\prime})^9=Id$ and then we find an element of order $9$ in $\mathrm{GL}_r(\Q )$. This is impossible: The minimal polynomial should divide $X^9-1=(X - 1)(X^2 + X + 1)(X^6 + X^3 + 1)$. As these $3$ are irreducible, it divides $(X-1)(X^2+X+1)$ and thus $A^3=1$.
\end{proof}
 
\begin{lem}\label{lem:PGL3bound3}
    Let $G\le \PGL_3(\Q)$. If $\mid G\mid =3^l$, then $l\le1$.
\end{lem}
 
\begin{proof}
    Let us assume, for contradiction, the existence of $G\subset\PGL _3(\Q )$ a group of order $3^l$ where $l> 1$. 
    Since $G$ is a $3-$group, from \cite[Proposition 3.9]{MR2330890}, the center $G$ of the group $G$ is not trivial and by Lemma \ref{element3}, except the identity, every element in $Z(G)$ is of order $3$. By Lemma \ref{lemm-29}, we can assume that $Z(G)$ contains the class $\beta$ of $B$, where 
    \[B=\left(\begin{matrix}
       0&0  & \lambda \\
          1&0  &0\\
           0&1 &0
    \end{matrix}\right),~\lambda\in\Q \setminus\{0\}.\]
    Let $\alpha\in G\setminus Z(G)$ be an element. By definition, the elements in the center of a group commute with all elements of the group $G$. So we need to find a matrix $A$ whose class $\alpha$ in $\PGL _3(\Q )$ has order $3$ and commutes with the class $\beta$. This implies that $\alpha$ fixes the set of $3$ fixed points of $B$. The three fixed points are in $\p^2(\C )$ and given by the classes of the three eigenvectors
\[v_1=\left(\begin{matrix}
    b^2\\b\\1
\end{matrix}\right),~v_2=\left(\begin{matrix}
    b^2 \omega\\b\omega^2\\1
\end{matrix}\right),~v_3=\left(\begin{matrix}
    b^2 \omega^2\\b\omega\\1
\end{matrix}\right),\]
where $\omega$ is the primitive cubic root of unity and $\lambda=b^3,~b\in\R $. Since one of the fixed points is real and the other two are not, this implies that the class $\alpha$ can not permute transitively the three fixed points, so either $v_1$ is fixed and $\alpha$ permute $v_2,v_3$, which is not possible as $\alpha$ is of order $3$, or every $v_i$ is fixed by $\alpha$ for $i=1,2,3$. So $Av_i=\alpha_i v_i$ where $\alpha_i\in\C $ for $i=1,2,3$. The matrix $A$ is conjugated in $\PGL _3(\C )$ to the class of the matrix
 \[C=\left(\begin{matrix}
       \alpha_1&0  & 0 \\
          0& \alpha_2 &0\\
           0&0 & \alpha_3
    \end{matrix}\right).\]
The class of $C$ has order $3$, so $ \alpha_1^3= \alpha_2^3= \alpha_3^3$. 
We have $\bar{A}=A$, so $\overline{Av_i}=\overline{\alpha_iv_i}\Rightarrow A\bar{v_i}=\bar{\alpha_i}\bar{v_i}$. Since $\bar{v_1}=v_1$, this implies that $\bar{\alpha_1}=\alpha_1$. The same for $\bar{v_2}=v_3$ implies that $\bar{\alpha_2}=\alpha_3$. We then have three possibilities for $\alpha_i$: which then only give three possibilities for $\alpha$, namely $id,~\beta$ and $\beta^2$
\begin{itemize}
    \item [a)]  $\alpha_2=\alpha_1\Rightarrow \alpha_3=\alpha_1\Rightarrow \alpha_1=\alpha_2=\alpha_3 \Rightarrow \alpha=Id$,
    \item [b)]  $\alpha_2=\omega\alpha_1\Rightarrow \alpha_3=\omega^2\alpha_1 \Rightarrow \alpha=\beta$,
    \item [c)]  $\alpha_2=\omega^2\alpha_1\Rightarrow \alpha_3=\omega\alpha_1 \Rightarrow \alpha=\beta^{-1}$.
 \end{itemize}
We find that $\alpha\in \langle \beta\rangle \subseteq Z(G)$ is a contradiction.
\end{proof}
\begin{lem}\label{lemm-26}
    If $G\subset\PGL _3(\Q )$ is a finite subgroup, then $\mid G\mid $ divides $24$.
\end{lem}
\begin{proof}
Let $G\subset\PGL _3(\Q )$ be a subgroup of finite order $n> 1$. let $p$ be a prime number divides $n$, then by Lemma \ref{Lemma:pgroup}, $p\le 3$. So we can write $n=2^{r_1}\cdot 3^{r_2}$. By Remark \ref{remk-04}, $r_1\le 3$. By Lemma \ref{lem:PGL3bound3}, $r_2\le 1$. This implies that $n$ divides $2^3\cdot 3^1=24$.
\end{proof}

\begin{ex} \label{ex-07}   
Let $G$ be the finite group in $\PGL _3(\Q )$ generated by the following three elements,
\[\rho=\left(\begin{matrix}
       - 1&0  & 0 \\
          0& - 1 &0\\
           0&0 & 1
    \end{matrix}\right),~~\sigma=\left(\begin{matrix}
      0&1  & 0 \\
          1& 0 &0\\
           0&0 &1
    \end{matrix}\right),~~\tau=\left(\begin{matrix}
       0&1  & 0 \\
          0& 0 &1\\
           1&0 &0
    \end{matrix}\right).\]
In this case, we have $\rho^2=\sigma^2=\tau^3=Id$, and $\mid G\mid =24$. Moreover, $G\cong (\Z /2\Z )^2\rtimes \sym_3$ 
\end{ex}
  \section{\texorpdfstring{Del Pezzo surfaces of degree $4$}{Del Pezzo surfaces of degree 4}}

Let $S$ be a Del Pezzo surface of degree $4$ defined over $\Q $. Working over the field $\C $ of complex numbers, $S$ can be viewed as the blow-up of five points in general position (no 3 collinear). Writing $A_1 = [1 : 0 : 0],~A_2 = [0 : 1 : 0],~A_3 = [0 : 0 : 1],~
A_4 = [1 : 1 : 1]$ and $A_5 = [a: b: c]$ the five points, the exceptional divisors on $S$ are:
\begin{itemize}
    \item $E_1 = \pi^{-1}(A_1), \ldots ,E_5 = \pi^{-1}(A_5)$, the $5$ pull-backs of the points $A_1,\ldots ,A_5$.
  \item $l_{i j} =  \pi^{-1}(A_i A_j)$ for $1 \le  i, j \le  5,~i\neq j$, the $10$ strict pull-backs of the lines through $2$ of the $A_i’s$.
   \item $C=\pi^{-1}(\mathcal{C})$, the strict pull-back of the conic $\mathcal{C}$ of $\p^2$ through $A_1,A_2, \ldots ,A_5$.
\end{itemize}
\begin{defi}\label{def-16}
   A pair $\{A, B\}$ of classes of $\C-$divisors where $A,B \in Pic_C(S)$ is an exceptional pair if $A$ and $B$ are $\C -$divisors, $A^2 = B^2 = 0$ and $(A + B) = -K_S$. This implies that $AB = 2$ and $AK_S = BK_S = -2$.
\end{defi}    
\begin{lem}\cite[Lemma $9.11$]{MR2486798}\label{lemm-14}
    \begin{itemize}
        \item There are $5$ exceptional pairs on the surface $S_{\C }$, namely $\{L-E_i, 2L-\Sigma_{j\neq i} E_j\} =\{L-E_i,-K_S-L+E_i\}$, for $i = 1, \ldots , 5.$
 \item Geometrically, these correspond to the lines through one of the five blown-up points, and the conics through the other four points.
    \end{itemize}
\end{lem}   
 
The action of $\Aut_{\C }(S)$ on the exceptional pairs gives rise to a homomorphism
$\rho : \Aut_{\C }(S) \rightarrow \sym_5$. Denoting its image and kernel respectively by $H_S$ and $G_S$, then we have 
 
\begin{lem}\cite[Lemma $9.11$]{MR2486798}\label{lemm-15}
    \begin{itemize}
\item The following exact sequence splits.
\begin{equation}\label{exact2:DelPezzo4}
    1\rightarrow G_{S}\rightarrow \Aut_{\C }(S)\rightarrow H_{S}\rightarrow 1
\end{equation}
\item The homomorphism $\gamma:G_S\rightarrow(\F _2)^5$, induces an isomorphism $G_{S}\rightarrow\{(a_0,a_1,a_2,a_3,a_4)\in(\F _2)^5\mid \Sigma_{i=0}^{4} a_i=0\}\cong(\F _2)^4$.
\item  $H_{S}\cong\{h\in \PGL _3(\C )\mid h(\{A_1,A_2,A_3,A_4,A_5\})=\{A_1,A_2,A_3,A_4,A_5\}\}\subset \sym_5$.
\item $\Aut_{\C }(S)\cong G_{S}\rtimes H_{S}$, where the group $H_{S}\subset\sym_5$ acts on $G_{S}$ by permuting the $5$ coordinates of its elements.
\end{itemize}
\end{lem}
 
\begin{lem}\label{lemm-16}
Let $S$ be a Del Pezzo surface of degree $4$ defined over $\Q $ obtained by the blow-up of five points of $\p^2(\Q )$ in general position, then we have $\mid \Aut(S)\mid \le 32$.
\end{lem}
 
\begin{proof}
Let $S$ be a Del Pezzo surface of degree $4$ defined over $\Q $ and this surface $S$ is obtained by the blow-up of five points of $\p^2(\Q )$ in general position. We may assume that the five points are $A_1 = [1 : 0 : 0],~A_2 = [0: 1 : 0],~A_3 = [0 : 0: 1],~A_4 = [1: 1: 1]$ and $A_5 = [a: b: c]$, where $A_5$ is a point not aligned with any two of the others, and $a,b,c\in \Q \setminus\{0\}$. Since the five points are defined over $\Q $ and $\Q \subseteq\C $, then by Lemma \ref{lemm-15}, we have $\Aut(S_{\Q })\cong (\Z /2\Z )^4\rtimes H_{S}$ and the group $H_{S}\subset\sym_5$ acts on $G_{S}$ by permuting the $5$ coordinates of its elements. It remains to study the group $H_S$. We prove that any nontrivial element of $H_S$ is a $2\times2-$cycle.
\begin{enumerate}
 
 \item Let $H_S$ contains a $4-$cycle, namely $(A_1 A_3A_2A_4)$. The automorphism is then $[x:y:z] \mapsto [y : y-z: y-x]$. The point $A_5$ must be a fixed point, so we have either $b=\mathrm{i}a$ where $\mathrm{i}=\sqrt{-1}$  or $\{b=a,c=0\}$, then we have a contradiction: it is clear in the first case as $b\notin \Q $ and in the second case, we have $[a:b:c]=[1:1:0]$ which is collinear with $A_1$ and $A_2$.
 
\item Let $H_S$ contains a $3-$cycle, namely $(A_1 A_2A_5)$. The automorphism is then $\varphi:[x : y : z] \mapsto [a^2y : -b^2x+aby : acy-ab z]$ for some $a,b,c$ in $\Q $ with $a\cdot b\neq0$. Since $\varphi$ sends $A_1$ on $A_2$, $A_2$ on $A_5$, and $A_5$ on $A_1$, then $A_3$ and $A_4$ must be fixed, so we have the equation $\varphi(A_4)=[a^2:-b(a-b):a(c-b)]$, so $a\in\{\omega b,\omega^2b\}$, where $\omega$ is the $3rd$ root of unity which is a contradiction.
 
\item Let $H_S$ contains a $3\times2-$cycle, namely $(A_1 A_2A_5)(A_3A_4)$. The automorphism is then $\varphi:[x : y : z] \mapsto [a^2y : -b^2x+aby : acy-ab z]$ for some $a,b,c$ in $\Q $ with $a\cdot b\neq0$. Since $\varphi$ sends $A_1$ on $A_2$, $A_2$ on $A_5$, $A_5$ on $A_1$, and then $A_4$ has to be sent on $A_3$, but $\varphi(A_4)=[a^2:-b(a-b):a(c-b)]\ne[0,0,1]$ for all $a,b,c\in\Q $, so $A_4$ can not be sent to $A_3$ which is a contradiction.
 
\item Let $H_S$ contains a $2-$cycle, namely $(A_1A_5)$. The automorphism is then has to fix the other three points $A_2,A_3$, and $A_4$. So it is given by $\varphi:[x : y : z] \mapsto [ax : bx+(a-b)y : cx+(a-c) z]$ for some $a,b,c$ in $\Q $ with $a\cdot b\cdot c\neq0$. Since $\varphi$ sends $A_1$ on $A_5$, then $A_5$ has to be sent on $A_1$, but $\varphi(A_5)=[a^2:2ab-b^2:2ac-c^2]=[1:0:0]$, so we have either $c=b$ and $a=b/2$ which is a contradiction: hence $A_5=[1,2,2]$ which is collinear with $A_1$ and $A_4$.
 
\item Let $H_S$ contains a $5-$cycle, namely $(A_1 A_4A_3A_2A_5)$. The automorphism is then $[x : y : z] \mapsto [a(x-y) : ax-b y-(a-b)z : ax-cy]$ for some $a,b,c$ in $\Q $ with $a\cdot b\neq0$. Since the automorphism sends $A_1$ on $A_4$, $A_4$ on $A_3$, $A_3$ on $A_2$, and $A_2$ on $A_5$, then the image of $A_5$ is $[a^2-ab:a^2-b^2-(a-b)c:a^2-bc]$ must be $A_1$, then $a^3-2a^2b-b^3=0$, assuming $a=t\cdot b$ implies that $b^3(t^{3}-2 t^{2}-1)=0$. then either $a=b=0$ or $t^{3}-2 t^{2}-1=0$, but the discriminant of the polynomial $t^{3}-2 t^{2}-1$ is $-59$, then polynomial $t^{3}-2 t^{2}-1$ has one real root and two complex conjugate roots, and then the real root is $x_0=(r^{\frac{2}{3}}+4 r^{\frac{1}{3}}+16)/(6 r^{\frac{1}{3}})$ where $r=172+12 \sqrt{177}$, which is not rational, so there is no non-trivial solution in $\Q $, which is a contradiction.\end{enumerate}

As $H_S$ only contains $2\times2-$cycle, it has order $1,~2$ or $4$. One then needs to remove the case where $\mid H_S\mid =4$ to finish the proof. In that case, $H_S$ would be, up to conjugation, the Klein group $\langle(A_1 A_2)(A_3 A_4), (A_1 A_3)(A_2 A_4)\rangle $  and thus would be generated by $[x:y:z]\mapsto[z-y:z-x:z]$ and $[x:y:z]\mapsto[y-z:y:y-x]$, so $A_5$ should be fixed by both elements. However, the only three points of $\p^2(\C)$ fixed by these two involutions are $[0:1:1],[1:0:1]$ and $[1:1:0]$, and any of these is collinear with two of the points $A_1,A_2,A_3,A_4$. This leads to a contradiction, that shows that $H_S$ has order $1$ or $2$, so $\Aut(S)$ has order $16$ or $32$.
\end{proof}
 
\begin{ex}\label{ex-06}
    Let $L=\mathbb{Q}[\omega]$, where $\omega$ is a primitive third root of unity. We consider the standard Del Pezzo surface of degree $6$  given by 
\[X_L:=\{([x_0:x_1:x_2],[y_0:y_1:y_2])\in \mathbb{P}^2_L \times \mathbb{P}^2_L\mid x_0y_0=x_1y_1=x_2y_2\}\]
as in Section~\ref{section:6}, and consider the action of $\Gal(L/\mathbb{Q})$ on $X_L$ given by 
\[([x_0:x_1:x_2],[y_0:y_1:y_2])\mapsto ([\bar{y_0}:\bar{y_1}:\bar{y_2}],[\bar{x_0}:\bar{x_1}:\bar{x_2}]).\]
This gives to $X$ a unique structure of a surface defined over $\mathbb{Q}$. It contains the $\mathbb{Q}$-rational point $([1:1:1],[1:1:1])$ and is thus rational \cite[Theorem 3.15]{MR225780}. We choose the two other $L$-rational points 
\begin{equation*}
    q_1=([1:\omega:\omega^2],[1:\omega^2:\omega]), q_2=([1:\omega^2:\omega],[1:\omega:\omega^2]),
\end{equation*}
denote by $Y$ the blow-up $Y$ of $X$ at the two points $q_1,~q_2$. We now prove that $\Aut(Y)$ is isomorphic to $(\mathbb{Z}/2\mathbb{Z})^4\rtimes \mathrm{Sym}_3$ and thus has $2^4\cdot 6=96$ elements.\\

Working over $L$, the surface $Y$ is the blow-up of the five points $p_1=[1:\omega:\omega^2]$, $p_2=[1:\omega^2:\omega]$, $p_3=[1:0:0]$, $p_4=[0:1:0]$and $p_5=[0:0:1]$. The automorphism $\Aut_L(Y)$ is then $(\mathbb{Z}/2\mathbb{Z})^4\rtimes H$, where $H\simeq \mathrm{Sym}_3$ is obtained by the permutations on the three coordinates. The Galois group corresponds to $(1,1,0,0,0)$, that commutes with all elements of $\Aut_L(Y)$, so $\Aut_{\mathbb{Q}}(Y)=\Aut_L(Y)$ is isomorphic to $(\mathbb{Z}/2\mathbb{Z})^4\rtimes \mathrm{Sym}_3$.
\end{ex}
 
\begin{lem}\label{lemm-25}
Let $S$ be a Del Pezzo surface of degree $4$ defined over $\Q$, then $\mid\Aut_{\Q}(S)\mid\le 96$.
\end{lem}
 
\begin{proof}
 Let us assume that $S$ is a Del Pezzo surface of degree $4$ defined over $\Q $. Let $L/\Q $ be the smallest Galois field extension such that the $(-1)-$curves on $S$ are defined over $L$. The surface $S$ has sixteen $(-1)-$curves defined over $L$ but not necessarily over $\Q $. By Lemma \ref{lemm-15}, we have $\Aut_{\C }(S)\cong G_S\rtimes H_S$, $G_S\cong (\F_2)^4$, and the group $H_S\subset \sym_5$ is isomorphic to either $\Z /2\Z ,~\Z / 4\Z ,~\sym_3$ or $D_5$ (the dihedral group of order $10$). Since the surface $S$ is defined over $\Q $, then $\Aut_{\Q }(S)\subseteq\Aut_{L}(S)\subseteq\Aut_{\C }(S)$. If $\mid \Aut_{\Q }(S)\mid > 96$, then $\Aut_{\Q }(S)\cong(\F _2)^4 \rtimes D_5$. But in this case, every element in the Galois group $\Gal(L/\Q )$ should commute with the whole group $D_5$. Since the only element of the group $(\F_2^4)\rtimes D_5$ that commutes with that whole group is the identity, then the Galois group acts trivially. It follows that every $(-1)-$curve is defined over $\Q $, hence by Lemma \ref{lemm-16}, we know that $\mid \Aut_{\Q }(S)\mid \le 32$ which is a contradiction, so $H_S$ can not be $D_5$.
\end{proof}
 
\section{\texorpdfstring{Del Pezzo surfaces of degree $3$}{Del Pezzo surfaces of degree 3}}
Let $S$ be a rational Del Pezzo surface of degree $3$ defined over $\Q $, we want to find the biggest automorphism group $G\subseteq\Aut_{\Q }(S)$. The classification of all automorphism groups of $S$ over $\C $ has been studied in \cite[Theorem 9.5.6]{MR2964027}, and they explained that the Del Pezzo surface of degree $3$ is isomorphic to a one contained in one of the $11$ families of smooth cubic surfaces given in the list.
\begin{rmk}\label{remk-05}
    In Theorem \ref{thmm-05}, since $M(4,2)=7,~M(4,3)=2,~M(4,5)=1,$ and $M(4,p)=0$ for all primes $p> 5$, then we have 
    \[M(4)=2^7\cdot 3^2\cdot5^1\cdot\prod_{p> 5} p^0=2^7\cdot 3^2\cdot5.\] 
    So the order of any finite subgroup $G$ in $\mathrm{GL}_4(\Q )$ divides $2^7\cdot 3^2\cdot5$.
\end{rmk}
 
\begin{lem}\label{lemm-30}
    Let $G$ be a finite subgroup of $\PGL _{n}(\Q )$ and let $F\in \Q [x_0,\ldots ,x_{n-1}]$ be an irreducible homogeneous polynomial such that $G$ preserves $F$. This means that for each $g$ in $G$ and each matrix $M $ of $ \mathrm{GL}_n(\Q )$ such that g is the class of $M$, the polynomial $F \circ M$ (obtained by acting on the coordinates linearly) is a multiple of $F$. Then, if $d=deg(F)$ and $n$ are coprime, there exists a finite subgroup of $\mathrm{GL}_{n}(\Q )$ that surjects to $G$, is twice bigger, and contains $-Id$.
\end{lem} 
 
\begin{proof}
     Let $g$ be any element of $G$ and let $M$ be the corresponding matrix. We have $F\circ M=\lambda\cdot F$ for some  $\lambda\in\Q $. Let $m$ be the order of $g$. Then, $M^m$ is a multiple of the identity, which gives $\mu\in\Q $ such that $M^m=\mu\cdot Id$. We write $a=det(M)$. As $M^m=\mu\cdot Id$, we apply the determinant on both sides and obtain $a^m=\mu^n$. Moreover, $F \circ M=\lambda\cdot F$, so $F\circ (M^i)=\lambda^i \cdot F$ for each $i\ge  1$. In particular, $F\circ (M^m)=\lambda^m\cdot F$. As $M^m=\mu\cdot Id$, we find $F\circ (\mu\cdot Id)=F\cdot\mu^d $ (recall that $d$ is the degree of $F$), so $\mu^d=\lambda^m$. By Lemma \ref{lemm-18}, there exist $c\in\Q^*$ such that $\mu=c^m$. Replacing $M$ with $M/c$ we do not change the class $g$ in $G$, but may assume that $M^m=Id$. This implies that $\mathrm{det}(M)^m=1$, so $\mathrm{det}(M)\in \{-1,1\}$. We can then lift any element of $g \in G$ to an element of $\mathrm{GL}_n(\Q )$ that has determinant $ 1$ or $-1$. This gives exactly two possibilities for the lift, namely $M$ or $-M$. The set of such matrices makes a finite group twice bigger than $G$. 
\end{proof}
\begin{rmk}\label{remk-02} 
 Let $H$ be a finite subgroup of $\PGL _{4}(\Q )$ and let $F\in \Q [x_0,\ldots ,x_{n-1}]$ be an irreducible homogeneous polynomial of degree coprime to 4. If $H$ preserves $F$, then by Remark \ref{remk-05} together with Lemma \ref{lemm-30}, $\mid H\mid $ has to divide $2^6\cdot 3^2\cdot5$.
\end{rmk}
 
\begin{ex}\label{ex-01}
    Let $A=\left(\begin{matrix}
       0&-1 &0&0 \\
         1& -1 &0&0 \\
0& 0 &1&0 \\
0& 0 &0&1 \\
    \end{matrix}\right) \mathrm{~and~}B=\left(\begin{matrix}
       1&0 &0&0 \\
         0& 1 &0&0 \\
          0& 0 &0&-1 \\
          0& 0 &1&-1 \\ 
    \end{matrix}\right)$, then the class $\alpha$ (resp. $\beta$) of $A$ (resp. $B$) in $\PGL_4(\Q)$ is of order $3$. Furthermore, the group $G$ generated by $\alpha$ and $\beta$ is isomorphic to $(\Z/3\Z)^2$, and
 \[G=\langle\alpha,\beta~\mid ~\alpha^3=\beta^3=id,\alpha\beta=\beta\alpha\rangle.\]
\end{ex}
 
\begin{lem}\label{lemm-28}
    Let $S$ be a rational Del Pezzo surface of degree $3$ defined over $\Q $, then $\mid \Aut_{\Q }(S)\mid \le  120$.
\end{lem}
\begin{proof}
 Let $S$ be a rational Del Pezzo surface defined over $\Q$. Since $\Aut_{\Q}(S)\subseteq\Aut_{\C }(S)$, by \cite[Theorem 9.5.6]{MR2964027}, we only need to study the automorphism groups $G\subset\Aut_{\C}(S)$ such that $\mid G\mid > 120$, hence we have only one type to study. If the surface is isomorphic to the Fermat surface $V (t_0^3+ t_1^3+t_2^3+t_3^3)$, then $\Aut_{\C }(S)\cong (\Z /3\Z )^3\times \mathfrak{S}_4$ where $\mathfrak{S}_4$ is the group of permutations, then $\mid G\mid$ has to divide $648=2^3\cdot3^4$. Since $G\subseteq\PGL _4(\Q )$, then by Remark \ref{remk-02}, $\mid G\mid $ has to divide $2^6\cdot3^2\cdot5$, hence $\mid G\mid \le 2^3\cdot3^2=48$. 
\end{proof}
 
\begin{ex}\label{ex-02}
    Let $S$ be a Del Pezzo surface that is isomorphic to the Clebsch diagonal surface $V(t_0^2\cdot t_1+ t_1^2\cdot t_3+t_2^2\cdot t_0+t_3^2\cdot t_2)$. One can embed the surface $S$ in $\p^4$ and show that $S$ is isomorphic to the following surface $X$ in $\p^4$: 
    \[\sum_{i=0}^{4}X_{i}=\sum_{i=0}^{4}X_{i}^3=0.\]
    This surface  $X$ contains $2$ disjoint $(-1)-$curves defined over $\Q$ , given by
  \[ \begin{array}{l}
          E_1=\{[0:a:-a:b:-b]\mid [a:b]\in \p^1\},\\
        E_2=\{[a:0:b:-a:-b]\mid [a:b]\in \p^1\}.\\
    \end{array}\]
By contracting these two $(-1)-$curves, there is thus a contraction morphism $\eta: X\rightarrow S_5$, where $S_5$ is a Del Pezzo surface of degree $5$ with at least $2$ rational points, so the surface $S_5$ is a rational surface. Indeed, this follows from \cite[Theorem 3.15]{MR225780}. Furthermore, the group $G\cong\mathfrak{S}_5$ has order $120$.
\end{ex}
  \section{\texorpdfstring{Automorphism group of $\p^1\times \p^1$ over $\Q$}{Automorphism group over Q}}
Let $G$ be a finite group in $\PGL _2(\Q )$. \cite[Theorem 4.2.]{MR2681719} shows that the group $G$ can be of size $12$ at most. Moreover, the $2-$groups are of order at most $2^2$ and the $3-$groups are of order at most $3$.
\begin{ex} \label{ex-04} Let $A=\left(\begin{matrix}
       2&-1  \\
          1& 1 
    \end{matrix}\right) \mathrm{~and~}B=\left(\begin{matrix}
       0&1  \\
          1& 0 
    \end{matrix}\right)$, then the class $\alpha$ (resp. $\beta$) of $A$ (resp. $B$) in $\PGL _2(\Q )$ is of order $6$ (resp. $2$). Furthermore, the group generated by $\alpha$ and $\beta$ is the dihedral group $D_6$, where
 \[D_6=\langle\alpha,\beta~\mid ~\alpha^6=\beta^2=(\beta\alpha)^2=id~\rangle.\]
\end{ex}
 
\begin{ex}\label{ex-03}\cite{MR2681719}
Let $G\subset\PGL_2(\Q)$ be a finite subgroup, then the group $H:=(G\times G)\rtimes\Z/2\Z$ is a finite subgroup of $\Aut(\p^1\times \p^1)$. Furthermore, if $G\cong D_6$, then the group $H$ is of order  $288$.
\end{ex}
 \begin{lem}\label{lemm-27}
     Any finite subgroup of $\Aut(\p^1_{\Q}\times \p^1_{\Q})$ is always of order at most $288$.
 \end{lem}
\begin{proof}
    Since every finite subgroup of $G\subset\PGL_2(\Q)$ is of order at most $12$, then the group $H:=(G\times G)\rtimes\Z/2\Z$ is a finite subgroup of $\Aut(\p^1_{\Q}\times \p^1_{\Q})$, and thus $H$ has at most $288$.
\end{proof}
  \section{\texorpdfstring{Del Pezzo surfaces of degree $2$}{Del Pezzo surfaces of degree 2}}
A Del Pezzo surface of degree $2$ defined over $\C $ is given by the blow-up of $7$ points of $\p^2$ in general position (no six of them are on a conic and no three are collinear). Let $A_1, \ldots , A_7 \in\p^2$ be the $7$ points and $\pi:S\rightarrow \p^2$ the blow-down. The system of cubics passing through the $7$ points gives a morphism $\eta:S\rightarrow \p^2$ whose linear system is $\mid - K_S\mid $ and which is a double covering ramified over a smooth quartic curve $\Gamma$ of $\p^2$. (See \cite{MR1802909, MR2066103, MR2641179}):
\begin{center}
    \begin{tikzcd}
S \arrow["\pi^{\prime}"]{d} \arrow["\eta"]{rrrd} \\
\p^2\arrow[dashed,"{C^3(A_1,..,A_7)}~~~~~"]{rrr}&&& \p^2.
\end{tikzcd}
\end{center}
\begin{lem}\label{lemm-22}
    Let $S$ be a Del Pezzo surface of degree $2$ defined over $\Q $. If $G\subset\Aut(S)$ is finite subgroup, then $\mid G\mid \le 48$. 
\end{lem}
 
\begin{proof}
    Let $S$ be a Del Pezzo surface of degree $2$ defined over $\Q $. By \cite[Theorem III.3.5]{MR1440180} and \cite[Corollary 3.54]{MR2062787}, over $\C $, there exists a smooth quartic curve $\Gamma$ of $\p^2$ given by the equation $F(x,y,z)=0$ where $F$ is homogeneous of degree $4$ such that the surface $S$ can be viewed in the weighted projected space $\p(2,1,1,1)$. The morphism to $\p^2$ is the anti-canonical morphism, thus it is defined over $\Q $, so the quartic curve is also defined over $\Q $. Since every automorphism of $S$ must leave the quartic $w=F(x,y,z)=0$ invariant, we have a homomorphism $\rho: \Aut(S)\rightarrow \Aut(\p^2,\Gamma)$. By \cite[Proposition 8.3.1]{MR2286582}, the exact sequence 
    \begin{equation}\label{Exact:deg2}
    1\rightarrow\langle\sigma\rangle\rightarrow \Aut_{\C }(S)\rightarrow \Aut_{\C }(\Gamma)\rightarrow 1
\end{equation}
splits, where $\sigma$ denotes the Geiser involution $\sigma(w,x,y,z)=(-w,x,y,z)$. (See \cite[Lemma 8.3.2]{MR2286582}). Furthermore, $\Aut_{\C }(S)\cong~\langle\sigma\rangle\times\Aut_{\C }(\p^2,\Gamma)$. (see \cite[Proposition 8.3.1]{MR2286582}). Since $\Aut_{\Q }(\p^2,\Gamma)\subset \PGL _3(\Q )$, by Lemma \ref{lemm-26}, $\Aut_{\Q }(\p^2,\Gamma)$ is of size $24$ at most. This implies to $\mid \Aut_{\Q }(S)\mid =2\cdot\mid \Aut_{\Q }(\p^2,\Gamma)\mid \le 2\cdot24=48$.
\end{proof}
 
\begin{ex}\label{ex-08}
    Let $S$ be a Del Pezzo surface of degree $2$ defined in the weighted projective space $\p(2,1,1,1)$ by the equation
    \[3 w^{2}=-(x_1^{4}+x_2^{4}+x_3^{4})+5( x_1^{2} x_2^{2}+x_1^{2} x_3^{2}+ x_2^{2} x_3^{2}).\]
   This surface $S$ contains $3$ disjoint $(-1)-$curves, given by
  \[\begin{array}{l}
          E_1=\{w = x_1^{2}+x_1 x_2 +x_2^{2}, x_3 = -x_1 -x_2\},\\
        E_2=\{w = -x_1^{2}-x_1 x_2 -x_2^{2}, x_3 = x_1 +x_2\},\\
        E_3=\{w = x_1^{2}+2 x_1 x_2 -x_2^{2}, x_3 = 2 x_1 -x_2\}.
    \end{array}\]
By contracting these three $(-1)-$curves, there is thus a contraction morphism $\eta: S\rightarrow Y$, where $Y$ is a Del Pezzo surface of degree $5$, then $Y(\Q )\neq\emptyset$, so the surface $S$ is rational. Indeed, this follows from \cite[Theorem 3.15]{MR225780}. Furthermore, the automorphism group $\Aut(S)$ is generated by  
 \[\begin{array}{l}
    \rho_1: [w:x_1:x_2:x_3]\mapsto[a_1w:a_2\cdot x_1:a_3\cdot x_2:x_3], \\
    \rho_2: [w:x_1:x_2:x_3]\mapsto[w:x_{\tau(1)}:x_{\tau(2)}:x_{\tau(3)}],
 \end{array}\]
where $a_i\in\{\pm1\}$ for $i=1,..,3$ and $\tau\in\sym_3$, and then $\Aut(S)\cong(\Z /2\Z )^3\rtimes\sym_3$, and this implies to $\mid \Aut(S)\mid =8\cdot6=48$, this implies to $\mid \Aut(S)\mid =8\cdot6=48$.
\end{ex}
  \section{\texorpdfstring{Del Pezzo surfaces of degree $1$}{Del Pezzo surfaces of degree 1}}
A Del Pezzo surface $S$ of degree $1$ defined over $\C $ is given by the blow-up of $8$ points of $\p^2$, in general position. The linear system $\mid -2K_S\mid $ induces a degree $2$ morphism onto a quadric cone in $Q\subset\p^2$, ramified over the vertex $v$ of $Q$ and a smooth curve $C$ of genus $4$. Moreover, $C$ is the intersection of $Q$ with a cubic surface. (See \cite{MR1802909, MR2066103, MR2641179}).

\begin{lem}\label{lemm-23}
    Let $S$ be a Del Pezzo surface of degree $1$ defined over $\Q $. If $G\subseteq\Aut(S)$ is finite subgroup, then $\mid G\mid \le 12$. 
\end{lem}
 
\begin{proof}
    Let $S$ be a Del Pezzo surface of degree $1$ defined over $\Q $. By \cite[Theorem III.3.5]{MR1440180} and \cite[Corollary 3.54]{MR2062787}, up to a change of coordinates, we may assume that the surface $S$ has the equation
    \begin{equation}\label{Eq_3}
        w^2 - z^3 - F_4(x, y)z - F_6(x, y)=0,
    \end{equation}
in the weighted projective space  $\p(3,1,1,2)$, where $F_4$ and $F_6$ are forms of respective degree $4$ and $6$.  Any element of $\Aut(S_{\Q})$ is of the form
\[[w:x:y:z]\mapsto[\lambda\cdot w+p_3+z\cdot p_1:ax+by:cx+dy:\mu\cdot z+p_2],\]
where $p_i$ is a homogeneous polynomial of degree $i$ for $i=1,..,3$ and the other are constants in $\Q$, with $\lambda\cdot\mu\neq0$. Replacing in the equation we obtain
\[(\lambda\cdot w+p_3+z\cdot p_1)^2-(\mu\cdot z+p_2)^3-(\mu\cdot z+p_2)g_4-g_6,\]
where $g_i=f_i(ax+by,cx+dy)$. This should be a multiple of the Equation \ref{Eq_3}. The degree $1$ coefficient in $w$ is $2\cdot \lambda\cdot w\cdot (p_3+z\cdot p_1)$ and needs to be zero, so $p_3=0$ and $p_1=0$. Then, the degree $2$ coefficient of $z$ is $-3\cdot\mu^2\cdot p_2$ so needs to be zero. This implies that every automorphism is of the form 
\[[w:x:y:z]\mapsto[\lambda\cdot w:ax+by:cx+dy:\mu\cdot z].\]
Replacing in the Equation \ref{Eq_3}, we obtain
\[\lambda^2\cdot w^2-\mu^3z^3-\mu\cdot z \cdot g_4-g_6 =0,\]
where $g_i=f_i(ax+by,cx+dy)$. We find that $\lambda^2=\mu^3$, so by Lemma \ref{lemm-18}, there exists $u\in\Q^*$ such that $\lambda=u^3$ and $ \mu=u^2$. As $[w:x:y:z]=[u^3w:ux:uy:u^2z]$, we can replace $\lambda$ with $\lambda/u^3=1$ and $\mu$ with $\mu/u^2=1$. Every element is then of the form
\[[w:x:y:z]\mapsto[w:ax+by:cx+dy:z],\]
where $\left(\begin{array}{cc}
     a& b \\
    c & d
\end{array}\right)\in\GL_2(\Q)$. Moreover the matrix is unique, as if we multiply everything with $[s^3:s:s:s^2]$, then to obtain $s^2=s^3=1$ we need $s=1$. The group of automorphism is then a finite subgroup of $\GL_2(\Q)$. In $\GL_2(\Q)$, every finite subgroup divides $12$ (see \cite[Chap. IX, section 14]{MR340283}), then $\mid \Aut(S_{\Q})\mid \le 12$. Replacing again in the Equation \ref{Eq_3}, we see that $f_i(ax+by,cx+dy)=f_i$ for each $i$. The group of automorphism is then just 
\[\{ A\in \GL_2(\Q) \mid  f_i\circ A=f_i \text{ for each }i\}.\]
\end{proof}
 
\begin{ex}\label{ex-10}
Let $S$ be a Del Pezzo surface of degree $1$ defined over $\Q $, given by the equation
\[ w^2 - z^3 - F_4(x, y)\cdot z - F_6(x,y)=0,\]
where $F_4=\mu\cdot (x^2+xy+y^2)^2$,
\[F_6= \lambda\cdot(x^2y^4+2x^3y^3+x^4y^2)+(3x^5y+3xy^5-5x^3y^3+x^6+y^6),\]
and $\mu,\lambda\in\Q^*$. Then the automorphism group $\Aut_{\Q}(S)$ of $S$ is generated by the two elements 
\begin{eqnarray*}
    \alpha:&[w:x:y:z]\mapsto[w:-y:x+y:z],\\
   \beta:&[w:x:y:z]\mapsto[w:-x-y:y:z],
\end{eqnarray*}
where $\alpha^6=\beta^2=id$ and $\beta\alpha\beta=\alpha^{-1}$. Hence, the group $\Aut(S_{\Q})\cong D_6$ and it has $12$ elements. Moreover, If $\left\{\lambda = -{\frac{1}{5}}, \mu = -{\frac{6}{5}} \right\}$, this surface $S$ contains $4$ disjoint $(-1)-$curves, given by
  \[ \begin{array}{l}
          E_1=\{ w = x^{2} y +x \,y^{2}, z = -x^{2}-x y -y^{2}\},\\
          E_2=\{w = -x^{2} y -x \,y^{2}, z = -x^{2}-x y -y^{2}\},\\
        E_3=\{w = xy \left(x +y \right) , z = -x^{2}-x y -y^{2}\},\\
        E_4=\{{w = -xy(x + y), z = -x^2 - xy - y^2}\}.
    \end{array}\]
By contracting these four $(-1)-$curves, there is thus a contraction morphism $\eta: S\rightarrow Y$, where $Y$ is a Del Pezzo surface of degree $5$, then $Y(\Q )\neq\emptyset$, so the surface $S$ is rational. Indeed, this follows from \cite[Theorem 3.15]{MR225780}. 
\end{ex}
\bibliographystyle{alphadin}
\bibliography{reference.bib}
\end{document}